\documentclass[a4paper,12pt]{article}

\usepackage{mymacros}

\newcommand{\Perf}{\operatorname{Perf}}

\title{Cancellativization of dimer models}
\author{Charlie Beil, Akira Ishii, Kazushi Ueda}
\date{}
\begin{document}

\maketitle

\begin{abstract}
We show that any dimer model can be made cancellative
without changing the characteristic polygon.
\end{abstract}


\section{Introduction}
 \label{sc:introduction}

A {\em dimer model} is a bicolored graph on a real 2-torus $T$
giving a polygon division of $T$.
It is originally introduced in 1930s
as a model in statistical mechanics
\cite{Fowler-Rushbrooke},
and has been actively studied since then.
See e.g. a review by Kenyon
\cite{Kenyon_IDM} and references therein
for dimer models as statistical mechanical models.

More recently,
a new connection between dimer models and quivers
has been discovered by string theorists
(cf.~e.g.~\cite{Kennaway_BT}).
A dimer model encodes the information
of a quiver $\Gamma$ with relations,
and the resulting path algebra $\bC \Gamma$
is a {\em Calabi-Yau algebra} of dimension three
in the sense of Ginzburg
\cite{Ginzburg_CYA}
if and only if $\bC \Gamma$ is {\em cancellative}
(i.e., $a b = a c \ne 0$ for an arrow $a$
and a pair $(b, c)$ of paths implies $b = c$,
and similarly for $b a = c a$)
\cite{Broomhead,Mozgovoy-Reineke,Davison}.
One can also give a purely combinatorial condition
on a dimer model,
called the {\em consistency condition},
which is equivalent to the cancellation property
of the path algebra
if the dimer model is {\em non-degenerate}
\cite{Ishii-Ueda_CCDM,Bocklandt_CCDM}.
We say that a dimer model is {\em cancellative}
if it satisfies one (and hence all)
of these equivalent conditions.

With a dimer model,
one can associate two convex lattice polygons
called the {\em characteristic polygon}
and the {\em zigzag polygon}.
Here, a {\em lattice polygon} is the convex hull
of a finite lattice points on $\bR^2$.
Although these two polygons are different in general,
they coincide
if the dimer model is cancellative
\cite{Gulotta,Ishii-Ueda_DMSMCv1}.
We say that a polygon is {\em non-degenerate}
if it has an interior point.

It is easy to make a dimer model cancellative
without changing the zigzag polygon:

\begin{theorem} \label{th:same_zigzag}
If the zigzag polygon of a dimer model $G$ is non-degenerate,
then one can remove some edges and nodes from $G$
to obtain another dimer model $G'$,
which is cancellative with the same zigzag polygon as $G$.
\end{theorem}

As a corollary, one obtains the following:

\begin{corollary} \label{cr:zigzag<char}
For any dimer model,
the zigzag polygon is contained in the characteristic polygon.
\end{corollary}

It is more difficult
to make a dimer model consistent
without changing the characteristic polygon.
The main result in this paper states that
this is always possible:

\begin{theorem} \label{th:main}
If the characteristic polygon of a dimer model $G$ is non-degenerate,
then one can remove some edges from $G$
to obtain a cancellative dimer model $G'$
with the same characteristic polygon as $G$.
\end{theorem}

A dimer model is {\em strongly non-degenerate}
if every edge is contained in at least one corner perfect matching.
Any dimer model can be made strongly non-degenerate
without changing the characteristic polygon,
simply by removing edges not contained in any corner perfect matching.
The proof of Theorem \ref{th:main} also shows the following:

\begin{corollary} \label{cr:multiplicity-free}
If every corner perfect matching
in a strongly non-degenerate dimer model
is multiplicity-free,
then the zigzag polygon coincides with the characteristic polygon.
\end{corollary}

Although cancellativity is a strong condition
and there are many examples of non-cancellative dimer models
(see e.g. \cite{Davey-Hanany-Pasukonis};
in fact, we suspect that almost all dimer models
(in some random graph theoretic sense) are non-cancellative),
Theorems \ref{th:same_zigzag} and \ref{th:main}
shows the abundance of cancellative dimer models
among all dimer models.

This paper is organized as follows:
In Section \ref{sc:dimer},
we recall basic definitions on dimer models.
In Section \ref{sc:zigzag},
we recall the definition of a zigzag polygon
and prove Theorem \ref{th:same_zigzag}.
In Section \ref{sc:cancellativization},
we discuss an operation of cancellativization
which keeps the characteristic polygon fixed.
We first remove suitable edges
from the dimer model $G$
to obtain another dimer model $G'$
satisfying the following conditions:
\begin{itemize}
 \item
The characteristic polygon of $G'$ coincides
with that of $G$.
 \item
For any pair $(\frakc_1, \frakc_2)$ of adjacent corners
of the characteristic polygon of $G'$,
there are perfect matchings $D_1$ and $D_2$ of $G'$
such that
\begin{enumerate}[(i)]
 \item
the height changes of $D_1$ and $D_2$ give these corners;
$h(D_1) = \frakc_1$ and $h(D_2) = \frakc_2$, and
 \item
every connected component
of the symmetric difference $D_1 \triangle D_2$
is a zigzag path on $G'$.
\end{enumerate}
\end{itemize}
Then $G'$ has sufficiently many zigzag paths
to ensure that the zigzag polygon contains
the characteristic polygon.
On the other hand,
the characteristic polygon always contains
the zigzag polygon by Corollary \ref{cr:zigzag<char},
and hence they must coincide.
Now one can perform the operation
in Theorem \ref{th:same_zigzag}
and obtain a cancellative dimer model $G''$
with the same characteristic polygon as $G$.

{\em Acknowledgment}.
This project has been initiated
while C.~B. and K.~U. was attending the workshop
`{\em Linking representation theory,
singularity theory and non-commutative algebraic geometry}\,'
at Banff International Research Station,
whose hospitality is gratefully acknowledged.
This research is supported by Grant-in-Aid for Scientific Research (No.18540034)
and Grant-in-Aid for Young Scientists (No.24740043).

\section{Dimer models and quivers}
 \label{sc:dimer}

Let $T = \bR^2 / \bZ^2$
be a real two-torus
equipped with an orientation.
A {\em bicolored graph} on $T$
consists of
\begin{itemize}
 \item a finite set $B \subset T$ of black nodes,
 \item a finite set $W \subset T$ of white nodes, and
 \item a finite set $E$ of edges,
       consisting of embedded closed intervals $e$ on $T$
       such that one boundary of $e$ belongs to $B$
       and the other boundary belongs to $W$.
       We assume that two edges intersect
       only at the boundaries.
\end{itemize}
A {\em face} of a graph is a connected component
of $T \setminus \cup_{e \in E} e$.
The set of faces will be denoted by $F$.
A bicolored graph $G$ on $T$ is called a {\em dimer model}
if $G$ contains no univalent node
and every face $f \in F$ is simply-connected.

A {\em quiver} consists of
\begin{itemize}
 \item a set $V$ of vertices,
 \item a set $A$ of arrows, and
 \item two maps $s, t: A \to V$ from $A$ to $V$.
\end{itemize}
For an arrow $a \in A$,
the vertices $s(a)$ and $t(a)$
are said to be the {\em source}
and the {\em target} of $a$
respectively.
A {\em path} on a quiver
is an ordered set of arrows
$(a_n, a_{n-1}, \dots, a_{1})$
such that $s(a_{i+1}) = t(a_i)$
for $i=1, \dots, n-1$.
We also allow for a path of length zero,
starting and ending at the same vertex.
The {\em path algebra} $\bC Q$
of a quiver $Q = (V, A, s, t)$
is the algebra
spanned by the set of paths
as a vector space,
and the multiplication is defined
by the concatenation of paths;
$$
 (b_m, \dots, b_1) \cdot (a_n, \dots, a_1)
  = \begin{cases}
     (b_m, \dots, b_1, a_n, \dots, a_1) & s(b_1) = t(a_n), \\
      0 & \text{otherwise}.
    \end{cases}
$$
A {\em quiver with relations}
is a pair of a quiver
and a two-sided ideal $\scI$
of its path algebra.
For a quiver $\Gamma = (Q, \scI)$
with relations,
its path algebra $\bC \Gamma$ is defined as
the quotient algebra $\bC Q / \scI$.
Two paths $a$ and $b$ are said to be {\em equivalent}
if they give the same element in $\bC \Gamma$.

A dimer model $(B, W, E)$ encodes
the information of a quiver
$\Gamma = (V, A, s, t, \scI)$
with relations
in the following way:
The set $V$ of vertices
is the set of connected components
of the complement
$
 T \setminus (\bigcup_{e \in E} e),
$
and
the set $A$ of arrows
is the set $E$ of edges of the graph.
The directions of the arrows are determined
by the colors of the nodes of the graph,
so that the white node $w \in W$ is on the right
of the arrow.
In other words,
the quiver is the dual graph of the dimer model
equipped with an orientation given by
rotating the white-to-black flow on the edges of the dimer model
by minus 90 degrees.
The relations of the quiver are described as follows:
For an arrow $a \in A$,
there exist two paths $p_+(a)$
and $p_-(a)$
from $t(a)$ to $s(a)$,
the former going around the white node
connected to $a \in E = A$ clockwise
and the latter going around the black node
connected to $a$ counterclockwise.
Then the ideal $\scI$
of the path algebra is
generated by $p_+(a) - p_-(a)$
for all $a \in A$.

A {\em perfect matching}
on a dimer model $G = (B, W, E)$
is a subset $D$ of $E$
such that for any node $v \in B \cup W$,
there is a unique edge $e \in D$
connected to $v$.
A dimer model is {\em non-degenerate}
if for any edge $e \in E$,
there is a perfect matching $D$
such that $e \in D$.

A dimer model $G = (B, W, E)$ gives
a chain complex
$$
 0 \to \bZ^{F} \to \bZ^{E} \to \bZ^{B \sqcup W} \to 0
$$
computing the homology of $T$.
The orientation on a face comes from the standard orientation
of $T = \bR^2 / \bZ^2$,
and the orientation on an edge is such that
$\partial e = w - b$,
where $w$ and $b$ are the white and the black node
adjacent to the edge $e$.
A perfect matching $D \subset E$ gives a 1-chain
$\sum_{e \in D} e \in \bZ^E$
in this complex,
which will often be written as $D$
by abuse of notation.
By the definition of a perfect matching,
the difference of 1-chains
associated with a pair $(D, D')$ of perfect matchings
is a 1-cycle,
whose class in $H_1(T; \bZ)$ will be denoted by
$[D-D']$.
This class is equivalent to the class
$[D \triangle D']$
of a 1-cycle
supported on the symmetric difference
$
 D \triangle D' = (D \cup D') \setminus (D \cap D').
$
We have $D \triangle D' = - D' \triangle D$ as 1-cycles,
although the underlying sets are identical.

Let
$
 \la -,  - \ra : H_1(T; \bZ) \otimes H_1(T; \bZ) \to \bZ
$
be the intersection pairing.
The Poincar\'{e} dual of $[D \triangle D'] \in H_1(T; \bZ)$
is written as $h(D, D') \in H^1(T; \bZ)$,
and called the {\em height change} of $D$
with respect to the {\em reference matching} $D'$;
$$
 h(D, D')(C) = \la C, [D \triangle D'] \ra,
  \quad \forall C \in H_1(T; \bZ).
$$
We often suppress the reference matching from the notation
and write $h(D) = h(D, D')$.
We will use the isomorphism
$H^1(T; \bZ) \cong \bZ^2$
coming from the identification
$T = \bR^2 / \bZ^2$
to think of a height change
as an element of $\bZ^2$;
$h(D) = (h_x(D), h_y(D)) \in \bZ^2$.
The {\em characteristic polynomial} of $G$
is the generating function
$$
 Z(x, y)
  = \sum_{D \in \Perf(G)}
     x^{h_x(D)} y^{h_y(D)}
$$
for the height change,
which is a Laurent polynomial in two variables.
Its Newton polygon
$$
 \Conv
 \{ (h_x(D), h_y(D)) \in \bZ^2
      \mid \text{$D$ is a perfect matching} \}
$$
is called the {\em characteristic polygon}.
One clearly has
$
 h(D, D'') = h(D, D') - h(D'', D'),
$
so that the characteristic polygon will be translated
if one changes the reference matching.
A perfect matching $D$ is said to be
a {\em corner perfect matching}
if the height change $h(D)$ is at a corner
of the characteristic polygon.
The {\em multiplicity} of a perfect matching $D$
is the number of perfect matchings
whose height change is the same as $D$.

A perfect matching $D$ can be considered as a set of walls
which block some of the arrows.
A path $p$ on the quiver is said to be {\em allowed} by $D$
if $p$ does not contain any arrow contained in $D \subset E = A$.

With a perfect matching,
one can associate a representation of the quiver
with dimension vector $(1, \dots, 1)$
by sending any allowed path to $1$ and
other paths to $0$.
One can easily check
that this satisfies the relation of the quiver.
A perfect matching is said to be {\em simple}
if the associated quiver representation is simple,
i.e., has no non-trivial subrepresentation.
This is equivalent to the condition
that there is an allowed path
starting and ending at any given pair of vertices.


The main theorem of \cite{Ishii-Ueda_08} states that
when a dimer model is non-degenerate,
then the moduli space $\scM_{\theta}$
of $\theta$-stable representations of $\bC \Gamma$
of dimension vector $(1, \ldots, 1)$
is a smooth toric Calabi-Yau 3-fold
for a generic stability parameter $\theta$
in the sense of King \cite{King}.
A toric divisor in $\scM_{\theta}$ gives a perfect matching
so that the stabilizer group of the divisor
is determined by the height change of the perfect matching.

Although the following results are stated
in \cite[Proposition 8.2]{Ishii-Ueda_DMSMCv1}
for cancellative dimer models,
the proof works for any non-degenerate dimer model.

\begin{proposition} \label{pr:corner-matching}
The following hold for a non-degenerate dimer model:
\begin{enumerate}[(i)]
 \item \label{it:simple<->multiplicity-free}
A perfect matching $D$ is simple
if and only if it is multiplicity-free.
 \item \label{it:simple->corner}
A multiplicity-free perfect matching is a corner perfect matching.
\end{enumerate}
\end{proposition}

The dimer model $G_1$
in Figure \ref{fg:inconsistent-1}
shows that the converse
to Proposition \ref{pr:corner-matching}.(\ref{it:simple->corner})
does not hold in general.
The corresponding quiver is shown
in Figure \ref{fg:inconsistent-1-quiver}.
The set of perfect matchings
and the characteristic polygon
are shown in Figures \ref{fg:inconsistent-1-matching}
and \ref{fg:inconsistent-1-characteristic_polygon} respectively,
where the perfect matching $D_1$ is chosen
as the reference matching.
This example also shows that
one cannot obtain a cancellative dimer model
with the same characteristic polygon
simply by removing all arrows
not contained in any simple matchings;
if we perform this operation
on the dimer model $G_1$,
then the resulting dimer model $G_2$
shown in Figure \ref{fg:inconsistent-1-simples}
has a smaller characteristic polygon,
which coincides with the convex hull
of height changes of simple perfect matchings.

\begin{figure}[htb]
\begin{minipage}[t]{.5 \linewidth}
\centering
\input{inconsistent-1.pst}
\caption{The dimer model $G_1$}
\label{fg:inconsistent-1}
\end{minipage}
\begin{minipage}[t]{.5 \linewidth}
\centering
\input{inconsistent-1-quiver.pst}
\caption{The quiver $\Gamma_1$}
\label{fg:inconsistent-1-quiver}
\end{minipage}
\end{figure}

\begin{figure}[htb]
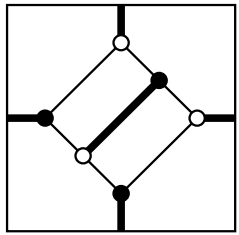
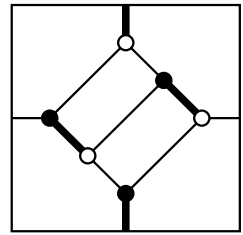
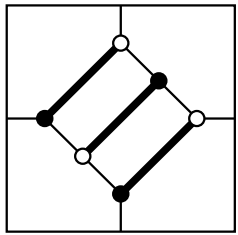
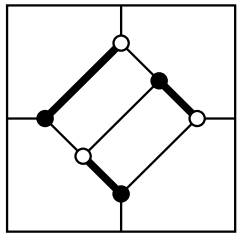
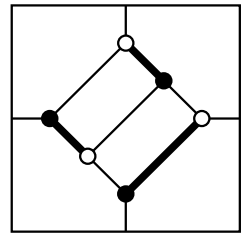

\centering
\subfigure[$h(D_1) = (0, 0)$]{
\input{inconsistent-1-matching1.pst}
\label{fg:inconsistent-1-matching1}}
\hspace{5mm}
\subfigure[$h(D_2) = (-1, 0)$]{
\input{inconsistent-1-matching2.pst}
\label{fg:inconsistent-1-matching2}}
\hspace{5mm}
\subfigure[$h(D_3) = (-1, -1)$]{
\input{inconsistent-1-matching3.pst}
\label{fg:inconsistent-1-matching3}}
\\
\subfigure[$h(D_4) = (0, -1)$]{
\input{inconsistent-1-matching4.pst}
\label{fg:inconsistent-1-matching4}}
\hspace{5mm}
\subfigure[$h(D_5) = (0, -1)$]{
\input{inconsistent-1-matching5.pst}
\label{fg:inconsistent-1-matching5}}
\hspace{5mm}
\subfigure[$h(D_6) = (0, -1)$]{
\input{inconsistent-1-matching6.pst}
\label{fg:inconsistent-1-matching6}}
\caption{Perfect matchings on $G_1$}
\label{fg:inconsistent-1-matching}
\end{figure}

\begin{figure}[htbp]
\begin{minipage}[t]{.5 \linewidth}
\centering
\input{inconsistent-1-characteristic_polygon.pst}
\caption{The characteristic polygon}
\label{fg:inconsistent-1-characteristic_polygon}
\end{minipage}
\begin{minipage}[t]{.5 \linewidth}
\centering
\input{inconsistent-1-simples.pst}
\caption{The union of simple matchings}
\label{fg:inconsistent-1-simples}
\end{minipage}
\end{figure}

\section{Zigzag polygon and cancellativity}
 \label{sc:zigzag}

A {\em zigzag path} is a path on a dimer model
which makes a maximum turn to the right on a white node
and to the left on a black node.
Note that it is not a path on a quiver.
We assume that a zigzag path does not have an endpoint,
so that we can regard a zigzag path
as a sequence $(e_i)_{i \in \bZ}$ of edges $e_i$ parameterized by $i \in \bZ$,
up to translations of $i$.
%
%
The homology class $[z]$ of a zigzag path
considered as an element of $\bZ^2$
will be called its {\em slope}.

Let $k$ be the number of zigzag paths.
Fix a zigzag path $z_1$,
and let $\{ z_i \}_{i=1}^k$ be the set of zigzag paths,
so that their slopes $([z_i])_{i=1}^k$ are cyclically ordered
starting from $[z_1]$.
Note that some of the slopes may coincide in general.
Define another sequence $(w_i)_{i=1}^r$ in $\bZ^2$ by
$w_0 = 0$ and 
$$
 w_{i+1} = w_i + [z_{i+1}]', \qquad i = 0, \dots, k-1,
$$
where $[z_{i+1}]'$ is obtained from $[z_{i+1}]$ by
rotating 90 degrees counter-clockwise.
Note that one has
$
 w_r = 0
$
since every edge is contained in exactly two zigzag paths
with different directions
and hence the homology classes of the zigzag paths add up to zero.
The convex hull of $(w_i)_{i=1}^r$ is called
the {\em zigzag polygon}.

Now we recall the definition of the consistency condition
for dimer models:

\begin{definition}[{\cite[Definition 5.2]{Ishii-Ueda_DMSMCv1}}] \label{df:consistency}
A dimer model is said to be {\em consistent} if
\begin{itemize}
 \item
there is no homologically trivial zigzag path,
 \item
no zigzag path on the universal cover
has a self-intersection, and
 \item
no pair of zigzag paths on the universal cover
intersect each other
in the same direction more than once.
\end{itemize}
\end{definition}

See \cite{Ishii-Ueda_CCDM,Bocklandt_CCDM} for more
on consistency conditions for dimer models.
The consistency condition is equivalent to cancellativity:

\begin{theorem}[{\cite[Theorem 1.1]{Ishii-Ueda_CCDM},
\cite[Theorem 6.2]{Bocklandt_CCDM}}]
A non-degenerate dimer model is consistent
if and only if the path algebra of the associated quiver with relations
is cancellative.
\end{theorem}

The characteristic polygon and the zigzag polygon coincides
for cancellative dimer models:

\begin{theorem}[{\cite[Theorem 3.3]{Gulotta},
 cf. also \cite[Corollary 8.3]{Ishii-Ueda_DMSMCv1}}]
 \label{th:gulotta}
For a consistent dimer model,
the characteristic polygon $\Delta$ coincides with the zigzag polygon
up to translation.
\end{theorem}


Now we prove Theorem \ref{th:same_zigzag}:

\begin{proof}[Proof of Theorem \ref{th:same_zigzag}]
If some zigzag path on the universal cover has a self-intersection,
then by removing all the edges at the self-intersection,
one obtains another bicolored graph on $T^2$
with the same zigzag polygon as the original dimer model.
Figure \ref{fg:zigzag_self-intersection} shows
an example of this operation.
If there is a connected component of the resulting graph
which is contained in a simply-connected domain in $T^2$,
then one can remove this connected component
without changing the zigzag polygon.
By removing all such components,
one obtains a dimer model
which has no zigzag path on the universal cover
with a self-intersection.
\begin{figure}[ht]
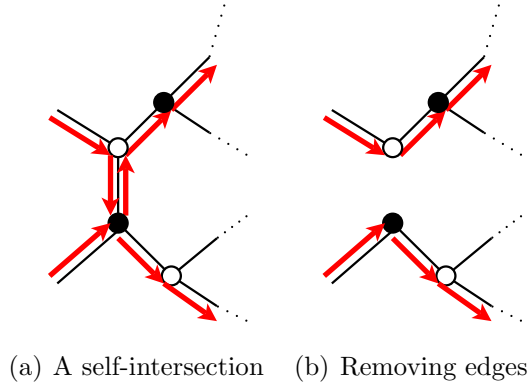

\centering
\subfigure[A self-intersection]
{\input{zigzag_self-intersection1.pst}
\label{fg:zigzag_self-intersection1}}
\subfigure[Removing edges]
{\input{zigzag_self-intersection2.pst}
\label{fg:zigzag_self-intersection2}}
\caption{A self-intersecting zigzag path
on the universal cover}
\label{fg:zigzag_self-intersection}
\end{figure}

If there is a homologically trivial zigzag path $z$,
then there are two cases;
either there is at least one edge inside the zigzag path $z$,
or there is no such edge.
If there is an edge inside the zigzag path,
take any zigzag path $w$ which intersects $z$.
Then $z$ and $w$ intersect in the same direction more than once,
and one can remove edges at the intersections
to obtain another dimer model.
If there are no edge inside the zigzag path $z$,
then every other node in $z$ is divalent,
and one can remove all these divalent nodes
and contract all other nodes to a single node.
Figure \ref{fg:trivial_zigzag} shows an example
of these operations.
\begin{figure}[ht]
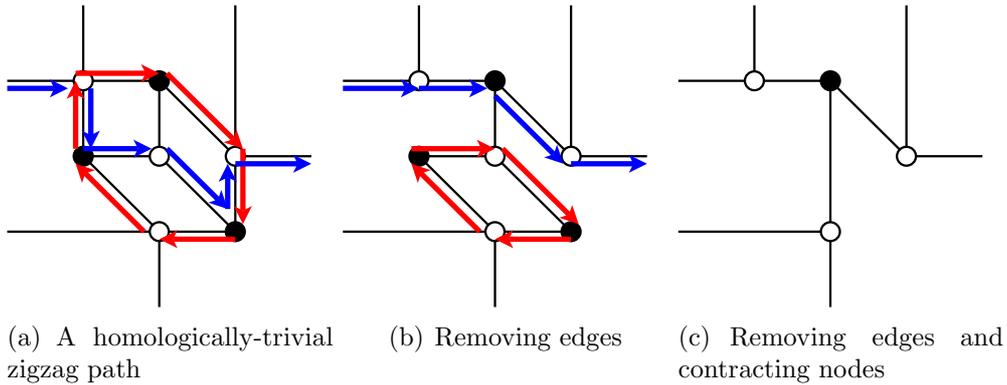

\centering
\subfigure[A homologically-trivial zigzag path]
{\input{trivial_zigzag1.pst}
\label{fg:trivial_zigzag1}}
\subfigure[Removing edges]
{\input{trivial_zigzag2.pst}
\label{fg:trivial_zigzag2}}
\subfigure[Removing edges and contracting nodes]
{\input{trivial_zigzag3.pst}
\label{fg:trivial_zigzag3}}
\caption{Homologically trivial zigzag paths}
\label{fg:trivial_zigzag}
\end{figure}

If there is a pair of zigzag paths on the universal cover
which intersect each other more than once
in the same direction,
choose any such pair of zigzag paths
and remove the edges
at a pair of consecutive intersections
of this pair of zigzag paths.
The resulting graph on the torus $T$
has the same set of slopes of zigzag paths,
and the non-degeneracy of the zigzag polygon implies
that this graph is still a dimer model
(i.e., there are no univalent node
and all the faces are simply-connected).

One can iterate these operations finitely many times
until the dimer model becomes cancellative.
\end{proof}

For example,
the dimer model $G_1$
in Figure \ref{fg:inconsistent-1}
has three zigzag paths as shown in Figure \ref{fg:inconsistent-1-zigzag}.
The corresponding zigzag polygon
is shown in Figure \ref{fg:inconsistent-1-zigzag_polygon}.
A pair of lifts of the zigzag path
shown in Figure \ref{fg:inconsistent-1-zigzag}.\ref{fg:inconsistent-1-zigzag1}
intersects in the same direction twice
on the universal cover
as shown in Figure \ref{fg:inconsistent-1-zigzag-intersection}.
Under the operation of `cancellativization'
in Theorem \ref{th:same_zigzag},
the pair of edges at these intersections will be removed,
and one obtains the dimer model
shown in Figure \ref{fg:inconsistent-1-simples}.

Corollary \ref{cr:zigzag<char} is an immediate consequence
of Theorem \ref{th:same_zigzag}:

\begin{proof}[Proof of Corollary \ref{cr:zigzag<char}]
The operation of cancellativization
in the proof of Theorem \ref{th:same_zigzag}
does not change the zigzag polygon,
but makes the characteristic polygon smaller in general.
Since characteristic polygon
and the zigzag polygon coincide
for a cancellative dimer model,
the zigzag polygon is smaller
than the characteristic polygon in general.
\end{proof}

\begin{figure}[htb]
\centering
\subfigure[A zigzag path with homology class $(1, -1)$]
{\input{inconsistent-1-zigzag1.pst}
\label{fg:inconsistent-1-zigzag1}} \hspace{5mm}
\subfigure[A zigzag path with homology class $(0, 1)$]
{\input{inconsistent-1-zigzag2.pst}
\label{fg:inconsistent-1-zigzag2}} \hspace{5mm}
\subfigure[A zigzag path with homology class $(-1, 0)$]
{\input{inconsistent-1-zigzag3.pst}
\label{fg:inconsistent-1-zigzag3}}
\caption{Zigzag paths on $G_1$}
\label{fg:inconsistent-1-zigzag}
\end{figure}

\begin{figure}[htbp]
\begin{minipage}[t]{.5 \linewidth}
\centering
\input{inconsistent-1-zigzag_polygon.pst}
\caption{The zigzag polygon of $G_1$}
\label{fg:inconsistent-1-zigzag_polygon}
\end{minipage}
\begin{minipage}[t]{.5 \linewidth}
\centering
\input{inconsistent-1-zigzag-intersection.pst}
\caption{Intersections of zigzag paths on $G_1$}
\label{fg:inconsistent-1-zigzag-intersection}
\end{minipage}
\end{figure}

\begin{remark}
The dimer model $G_1$ gives an example
where one can not obtain a cancellative dimer model
by the following simple operation:
\begin{itemize}
 \item
Take any generic stability parameter $\theta$
and contract all arrows
which does not vanish in any $\theta$-stable representations.
\end{itemize}
\end{remark}

\begin{proof}
Note that the height change $(0, -1)$ has multiplicity three.
Take a generic stability parameter
which makes the perfect matching $D_4$ stable.
Three other corner perfect matchings
$D_1$, $D_2$ and $D_3$ are simple,
so that they are stable for any stability parameter.
Now one can see that every arrow of $Q$
goes to zero in at least one $\theta$-stable representation
of dimension vector $(1, \ldots, 1)$.
\end{proof}

\section{Characteristic polygon and cancellativity}
 \label{sc:cancellativization}

We can always assume that a dimer model is non-degenerate
without changing the characteristic polygon:

\begin{proposition} \label{pr:non-degenerate}
Let $G$ be a dimer model
with a non-degenerate characteristic polygon.
Then one can remove some nodes and edges from $G$
to obtain a non-degenerate dimer model $G'$
with the same characteristic polygon as $G$.
\end{proposition}

\begin{proof}
Let $G''$ be the bicolored graph on $T$
whose set $E''$ of edges consists of edges of $G$
contained in at least one perfect matching of $G$,
and whose set of nodes consists of nodes of $G$
incident to at least one edge in $E''$.
Then $G''$ is clearly a non-degenerate graph.
In order to make $G''$ into a dimer model,
one removes all connected components of $G''$
having a simply-connected neighborhood in $T$.
The resulting graph $G'$ is a dimer model
(i.e. no node is univalent and
every connected component of $T \setminus G'$ is simply-connected)
having the same characteristic polygon as $G$.
\end{proof}

Let $G$ be a non-degenerate dimer model,
and consider a pair $(D_1, D_2)$ of perfect matchings.
Recall from Section \ref{sc:dimer}
that the homology class $[D_1 \triangle D_2]$ is Poincar\'{e} dual
to the height change $h(D_1, D_2)$.

\begin{lemma} \label{lm:components}
Let $D_1$ and $D_2$ be perfect matchings with
$
 v := [D_1 \triangle D_2]
  \ne 0 \in H_1(T, \bZ).
$
If the homology class of a connected component of $D_1 \triangle D_2$ 
is non-zero, it is one of the two primitive elements in $\bQ v \cap H_1(T, \bZ)$. 
Moreover, if either $D_1$ or $D_2$ is a corner perfect matching,
then it is the primitive element in $\bQ_+ v \cap H_1(T, \bZ)$.
\end{lemma}

\begin{proof}
Note that two cycles on a torus can be disjoint
only if their homology classes are proportional to each other.
Since $D_1 \triangle D_2$ is homeomorphic to the disjoint union of copies of $S^1$,
we obtain the first assertion.
Assume there is a connected component $w$ of $D_1 \triangle D_2$ whose homology
class is in $\bQ_- v$.
Then we can construct another perfect matching $D_3$
with $D_1 \triangle D_3 =w$.
The height change $h(D_1)$ of $D_1$ lies
on the line segment connecting $h(D_2)$ and $h(D_3)$,
so that $D_1$ is not a corner perfect matching.
By the same reasoning,
$D_2$ is not a corner perfect matching either.
\end{proof}

Fix a pair $(D_1, D_2)$ of corner perfect matchings
whose height changes are adjacent in the counter-clockwise order.

\begin{lemma} \label{lm:intersection}
For any perfect matching $D_3$
whose height change is not on the line segment
connecting $D_1$ and $D_2$,
one has
$$
 \la [D_1 \triangle D_2], [D_2 \triangle D_3] \ra > 0,
$$
where $\la - , - \ra$ denotes the intersection pairing on $H_1(T, \bZ)$.
\end{lemma}

\begin{proof}
This follows from the fact that $[D_1 \triangle D_2]$
is the Poincare dual of the relative height change $h(D_1, D_2)$
and the definition of the characteristic polygon.
\end{proof}

\begin{example}
Consider the dimer model $G_1$
given in Section \ref{sc:dimer}.
The cycles $[D_1 \triangle D_2]$ and $[D_2 \triangle D_3]$
are shown in Figures
\ref{fg:inconsistent-1-D1D2} and \ref{fg:inconsistent-1-D2D3} respectively,
which indeed satisfies
$$
 \la [D_1 \triangle D_2], [D_2 \triangle D_3] \ra > 0.
$$
\begin{figure}[ht]
\begin{minipage}[t]{.5 \linewidth}
\centering
\input{inconsistent-1-D1D2.pst}
\caption{The cycle $[D_1 \triangle D_2]$}
\label{fg:inconsistent-1-D1D2}
\end{minipage}
\begin{minipage}[t]{.5 \linewidth}
\centering
\input{inconsistent-1-D2D3.pst}
\caption{The cycle $[D_2 \triangle D_3]$}
\label{fg:inconsistent-1-D2D3}
\end{minipage}
\end{figure}
\end{example}

Lemma \ref{lm:intersection} can be rephrased as follows:

\begin{corollary} \label{cr:intersection}
If one goes along $D_1 \triangle D_2$ and
count the number of edges in $D_3$
connected to $D_1 \triangle D_2$ from the left,
then the number of edges of $D_3$ connected to white nodes is larger
than the number of those connected to black nodes.
The opposite inequality holds
if we count the number of edges of $D_3$
connected to $D_1 \triangle D_2$ from the right.
\end{corollary}

If some connected component $c$ of $D_1 \triangle D_2$
is homologically trivial,
then we can replace $D_1$
by another perfect matching $D_1'$
satisfying $D_1 \triangle D_1' = c$.
Then one has $h(D_1) = h(D_1') + [c] = h(D_1')$.
By continuing this process,
we may assume that $D_1 \triangle D_2$ does not have
any homologically trivial components.

Let $n$ be the number of
homologically non-trivial connected components of $D_1 \triangle D_2$.
We label these connected components as $\{ z_i \}_{i \in \bZ / n \bZ}$
in such a way that $z_i$ is right next to $z_{i-1}$ on the right
as shown in Figure \ref{fg:no_path}.

\begin{figure}[ht]
\centering
\input{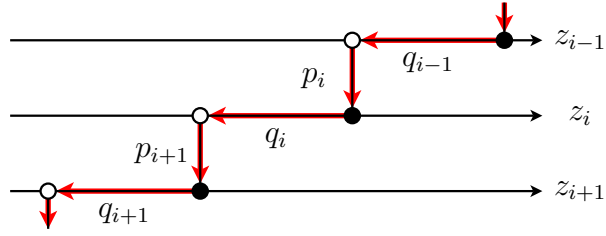}
\caption{The paths $z_i$, $p_i$ and $q_i$}
\label{fg:no_path}
\end{figure}

\begin{lemma} \label{lm:no_path}
There is a connected component $z_i$
with the following property:
\begin{itemize}
 \item
There is no path $p$ consisting of edges of $G$
satisfying the following conditions:
\begin{enumerate}[(i)]
 \item
The path $p$ is homeomorphic to the interval $[0,1]$.
 \item \label{it:non-crossing}
Every other edge of $p$ belongs to $D_1\cap D_2$.
 \item
The path $p$ connects a white node on $z_{i-1}$ to a black node on $z_i$.
 \item
The edge containing the white node $p \cap z_{i-1}$ is on the right of $z_{i-1}$
and the edge containing the black node $p \cap z_i$ is on the left of $z_i$.
\end{enumerate}
\end{itemize}
\end{lemma}

\begin{proof}
Note that the condition (\ref{it:non-crossing}) implies
that $p$ can not cross $z_j$ for any $j \in \bZ / n \bZ$.
Assume for contradiction that the assertion of Lemma \ref{lm:no_path} is false.
Then for each $i \in \bZ / n \bZ$,
there is a path $p_i$ satisfying the conditions above
as shown in Figure \ref{fg:no_path}.
Let $q_i$ be the part of $z_i$ which starts at the black node $z_i \cap p_i$
and goes backward (with respect to the orientation of $z_i$) to the white node $z_i \cap p_{i+1}$.
We can consider the path $y=\bigcup_i (p_i \cup q_i)$ which starts at $p_1 \cap z_0$,
goes along $p_1$ to $p_1 \cap z_1$,
then goes along $q_1$ to $p_{2} \cap z_1$,
then goes along $p_2$ to $p_2 \cap z_2$, and so on.
Then every other edge of $y$ belongs to $D_2$,
so that we can construct a perfect matching $D_3$ with $D_2 \triangle D_3 = y$.
Then every edge of $D_3$ connected to $D_1 \triangle D_2$ from the left
is connected to a black node.
This contradicts Corollary \ref{cr:intersection},
and Lemma \ref{lm:no_path} is proved.
\end{proof}

We say that a path $z$ is {\em zigzag at white nodes}
if there is no edge of $G$
connected to a white node on $z$ from the right.

\begin{lemma}\label{lm:zigzag_white}
Let $z_i$ be a connected component of $D_1 \triangle D_2$
with the property in Lemma \ref{lm:no_path}.
Then there are perfect matchings $\Dbar_1$ and $\Dbar_2$
with the same height changes as $D_1$ and $D_2$ respectively
such that
$$
 \Dbar_1 \triangle \Dbar_2 = \zbar_i \cup \bigcup_{j \ne i} z_j,
$$
where $\zbar_i$ is zigzag at white nodes.
\end{lemma}

\begin{proof}
We may assume $i=0$ without loss of generality.
Assume that an edge $e$ is connected
to a white node $w$ on $z_0$ from the right,
and take a perfect matching $D$ containing $e$.
Note that $D_1$ and $D_2$ coincide on the strip between $z_0$ and $z_1$.
The connected component $q$ of $D \triangle (D_1\cap D_2)$ containing $e$
forms an arc starting from the white node $w$
and ends at either $z_0$ or $z_1$.
The node at the intersection of $q$ with $z_0$ or $z_1$
other than $w$ must be a black,
and we will call it $b$.
Then the property in Lemma \ref{lm:no_path} implies that
$b$ must be on $z_0$.
For one of the two connected components of $z_0 \setminus \{ b, w \}$,
which we will call $q'$,
the union $q \cup q'$ forms a homologically trivial cycle.

Let us first consider the case
when $q'$ goes from $w$ to $b$
along the direction of $z_0$
as shown in Figure \ref{fg:zigzag_white1}.
Then we can take a perfect matching $D_2'$ with $D_2 \triangle D_2' = q \cup q'$. 
The resulting perfect matching $D_2'$ has the same height change as $D_2$, and
the connected components of $D_1 \triangle D_2'$ are $z_i$ ($i \ne 0$) and
$z_0':=(z_0 \setminus q') \cup q$.

\begin{figure}
\begin{minipage}[t]{.5 \linewidth}
\centering
\input{zigzag_white1.pst}
\caption{The edge $e$ and the paths $q$ and $q'$}
\label{fg:zigzag_white1}
\end{minipage}
\begin{minipage}[t]{.5 \linewidth}
\centering
\input{zigzag_white2.pst}
\caption{The paths $p$, $r$, and $z_0'$}
\label{fg:zigzag_white2}
\end{minipage}
\end{figure}

We claim that $z_0'$ also has the property in Lemma \ref{lm:no_path}.
Assume for contradiction
that there is a path $p$ satisfying the conditions in Lemma \ref{lm:no_path} for $z_0'$
as shown in Figure \ref{fg:zigzag_white2}.
Let $r$ be the part of $q$
starting from the white node $w$ and goes along $z_0'$
until it meets the white node at $p \cap z_0'$.
Let further $p'$ be the path obtained by concatenating $r$ and $p$.
Then $p'$ satisfies the conditions in Lemma \ref{lm:no_path} for $z_0$,
which is a contradiction.
Hence $z'_0$ has the property in Lemma \ref{lm:no_path}.

In the case where $q'$ goes from $b$ to $w$,
we can replace $D_1$ by $D_1'$
with $D_1 \triangle D_1' = q \cup q'$.
By the same argument as above,
one can show that
$D_1' \triangle D_2=z'_0 \cup \bigcup_{i \ne 0}z_i$
and $z_0'$ has the property in Lemma \ref{lm:no_path}.

Note that $z_0'$ is `closer' to $z_1$ than $z_0$.
If $z_0'$ is not zigzag at a white node,
then we can repeat the same operation.
Since there are only finitely many edges between $z_0$ and $z_1$,
this process terminates in finitely many steps, and
one obtains desired perfect matchings $\Dbar_1$ and $\Dbar_2$.
\end{proof}

So far,
we have shown the existence of a connected component $z_i$
in $D_1 \triangle D_2$
which is zigzag at white nodes.
In Lemma \ref{lm:remove_e} below,
we show that
if $z_i$ is not zigzag at a black node
by some edge $e$ of $G$,
then we can remove the edge $e$
without changing the characteristic polygon.

\begin{lemma} \label{lm:remove_e}
Assume that $z_i$ is zigzag at white nodes.
If $e$ is an edge connected to a black node $b$ on $z_i$ from the left of $z_i$,
then one can remove $e$ without changing the characteristic polygon.
\end{lemma}

\begin{proof}
It suffices to show that
for any perfect matching $D$ containing $e$,
there is another perfect matching $D'$
with the same height change as $D$
not containing $e$.
One may assume that the height change $h(D)$ of $D$ is
not on the line segment between $h(D_1)$ and $h(D_2)$.

Let $y$ be the connected component of $D_2 \triangle D$ containing $e$.
If $y$ is homologically trivial,
then take the perfect matching $D'$ such that $D' \triangle D=y$.
The matching $D'$ has the same height change as $D$
and does not contain $e$.
Hence we may assume that $y$ is homologically non-trivial.

Choose a lift $\btilde$ of the node $b$
to the universal cover $\bR^2 \to T$
and let $\ztilde_i$ and $\ytilde$ be the lifts of $z_i$ and $y$
containing $\btilde$ respectively.
Lemmas \ref{lm:components} and \ref{lm:intersection} imply
$
 \la z_i, y \ra > 0,
$
so that $\ytilde$ first comes from the right of $\ztilde_i$,
intersects $\ztilde_i$ several times,
and goes away to the left of $\ztilde_i$.
Hence there must be a white node $\wtilde \in \ztilde_i \cap \ytilde$
such that the direction of $\ytilde$ is from $\wtilde$ to $\btilde$.
We assume $\wtilde$ is the nearest to $\btilde$
in the part of $\ytilde$ before $\btilde$.
Let $b$ and $w$ be the images on the torus $T$
of $\btilde$ and $\wtilde$ respectively.

\begin{figure}[ht]
\centering
\input{remove_e1.pst}
\caption{The case when $\ztilde_i$ goes from $\wtilde$ to $\btilde$}
\label{fg:remove_e1}
\end{figure}

First we discuss the case
when $\ztilde_i$ goes from $\wtilde$ to $\btilde$.
Figure \ref{fg:remove_e1} shows the paths $y$ and $z_i$
on the torus $T$.
When we travel from $b$ along $y \subset D_2 \triangle D$,
the next edge $e_1$ is on $z_i \subset D_1 \triangle D_2$,
and the direction of $y$ is opposite to that of $z_i$ on that edge.
Then the next node $w_1$ is a white node on $y \cap z_i$.
Since $z_i$ is zigzag at white nodes,
the path $y$ cannot escape to the right of $z_i$,
and the next edge $e_2$ in $y$ either goes to the left of $z_i$
or on the path $z_i$.

If $y$ goes to the left of $z_i$,
then $y$ must eventually intersect $z_i$ again
since $y$ is an embedded circle in $T$.
If $e_2$ is on the path $z_i$,
then $e_2 \in D_1 \cap D$
and the next edge $e_3$ on the path $y$
is in $D_2$.
By continuing in this way,
one sees that $y$ must be contained
in the simply connected open subset $U$ of $T$
bounded by the parts of $y$ and $z_i$ between $b$ and $w$.
This implies that $y$ is homologically trivial,
which contradicts our assumption.

\begin{figure}[ht]
\centering
\input{remove_e2.pst}
\caption{The case when $\ztilde_i$ goes from $\wtilde$ to $\btilde$}
\label{fg:remove_e2}
\end{figure}

Hence the path $\ztilde_i$ goes from $\btilde$ to $\wtilde$
as shown in Figure \ref{fg:remove_e2}.
Let $e_1$ be the edge in $D_1$
incident to the node $b$.
By the definition of $z_i$,
the other node $w_1$ of $e_1$ is on the path $z_i$.
Take the edge $e_2$ in $D$
incident to $w_1$.
Since $z_i$ is zigzag on white nodes,
$e_2$ is either on $z_i$ or goes to the left of $z_i$.
If $e_2$ is on $z_i$,
then let $e_3$ be the edge in $D_1$
incident to the other node $b_2$ of $e_2$.
If $e_2$ goes to the left of $z_i$,
then continue $e_2$
along the connected component $y_1$ of $D_2 \triangle D$
containing $e_2$.
Then $y_1$ must eventually intersect $z_i$
at a black node,
which we will call $b_2$.
Let $e_3$ be the edge of $D_1$
incident to $b_2$.

By continuing in this way,
one can find a path $c$ from $b$ to $w$
which consists of parts of $z_i$ or $D_2 \triangle D$.
By concatenating $y$ with $c$,
one obtains a homologically trivial path on $G$
such that every other edge belongs to $D$.
Then the perfect matching $D'$
such that $D \triangle D' = y \cup c$
has the same height change as $D$
and does not contain $e$.
This concludes the proof of Lemma \ref{lm:remove_e}.
\end{proof}

\begin{lemma} \label{lm:zigzag}
Let $\frakc_1$ and $\frakc_2$ be adjacent corners of the characteristic polygon of $G$.
We can remove some edges
from $G$
to obtain a dimer model $G'$ such that
\begin{itemize}
 \item
the characteristic polygon of $G'$ coincides with that of $G$, and
 \item
there are perfect matchings $D_1$ and $D_2$ of $G'$
such that $h(D_1) = \frakc_1$, $h(D_2) = \frakc_2$ and
$D_1 \triangle D_2$ consists of zigzag paths.
\end{itemize}
\end{lemma}

\begin{proof}
First choose arbitrary perfect matchings
with height changes $\frakc_1$ and $\frakc_2$ respectively.
Take a path $z_i$ satisfying the property in Lemma \ref{lm:no_path}.
We may assume $i=0$ without loss of generality.
By Lemma \ref{lm:zigzag_white},
we can assume that $z_0$ is zigzag at white nodes
by replacing $D_1$ and $D_2$ if necessary.
If $z_0$ is not zigzag at some black node,
then one can use Lemma \ref{lm:remove_e}
to remove the edge $e$
which makes $z_0$ not zigzag at that node.
After iterating this operation finitely many times,
we can turn $z_0$ into a zigzag path.

Now the property in Lemma \ref{lm:no_path} holds for $z_{-1}$,
since a path $p$ satisfying the conditions should be connected
to a black node on $z_0$ from the left of $z_0$,
which is impossible since $z_0$ is a zigzag path.
Then we can repeat the same process
to turn $z_{-1}$ into a zigzag path.

By successively performing this operation,
we can turn all $z_i$ into zigzag paths.
\end{proof}

Now we can prove Theorem \ref{th:main}:

\begin{proof}[Proof of Theorem \ref{th:main}]
We can use Lemma \ref{lm:zigzag} repeatedly
to obtain another dimer model $G''$
such that
\begin{itemize}
 \item
the characteristic polygon of $G''$ coincides with that of $G$, and 
 \item
for any pair $(\frakc_1, \frakc_2)$
of adjacent corners of the characteristic polygon,
there are perfect matchings $D_1$ and $D_2$ of $G''$
such that $h(D_1) = \frakc_1$, $h(D_2) = \frakc_2$ and
$D_1 \triangle D_2$ consists of zigzag paths.
\end{itemize}
Zigzag paths constituting $D_1 \triangle D_2$
for pairs $(D_1, D_2)$ of adjacent corner perfect matchings
ensure that the zigzag polygon is at least as large as the characteristic polygon.
Then Corollary \ref{cr:zigzag<char} shows
that the zigzag polygon and the characteristic polygon of $G''$ coincide.
Now we can apply Theorem \ref{th:same_zigzag} to $G''$
to obtain a cancellative dimer model $G'$,
whose zigzag polygon is the same as that of $G''$.
Since $G'$ is cancellative,
the characteristic polygon of $G'$ coincides with its zigzag polygon,
which is the same as the characteristic polygon of $G$.
\end{proof}

Corollary \ref{cr:multiplicity-free} is an immediate consequence
of the proof of Theorem \ref{th:main}:

\begin{proof}[Proof of Corollary \ref{cr:multiplicity-free}]
The proof of Lemma \ref{lm:zigzag_white} shows that
if the connected component $z_i$ of $D_1 \triangle D_2$ with the property
in Lemma \ref{lm:no_path} is not zigzag at a white node,
then at least one of $D_1$ or $D_2$ has multiplicity.
In other words,
$z_i$ is zigzag at white nodes,
if both $D_1$ and $D_2$ are multiplicity-free.

On the other hand,
the proof of Lemma \ref{lm:remove_e} shows
that if a component $z$ of $D_1 \triangle D_2$ is
zigzag at white nodes but not zigzag
at a black node
by an edge $e$,
then any perfect matching containing $e$
has a multiplicity.
This cannot be the case
if the dimer model is strongly non-degenerate
and all the corner perfect matchings are multiplicity-free.

It follows by the argument in Lemma \ref{lm:zigzag} that
for a strongly non-degenerate dimer model,
the symmetric difference $D_1 \triangle D_2$
of a pair $(D_1, D_2)$ of perfect matchings,
whose height changes are adjacent corners
of the characteristic polygon,
consists of zigzag paths.
This implies that the zigzag polygon
is at least as large as (and hence coincides with)
the characteristic polygon,
and Corollary \ref{cr:multiplicity-free} is proved.
\end{proof}

\bibliographystyle{amsalpha}
\bibliography{bibs}

\noindent
Charlie Beil

Simons Center for Geometry and Physics,
State University of New York,
Stony Brook,
NY 11794-3636,
USA

{\em e-mail address}\ : \ cbeil@scgp.stonybrook.edu

\ \vspace{0mm} \\

\noindent
Akira Ishii

Department of Mathematics,
Graduate School of Science,
Hiroshima University,
1-3-1 Kagamiyama,
Higashi-Hiroshima,
739-8526,
Japan

{\em e-mail address}\ : \ akira@math.sci.hiroshima-u.ac.jp

\ \vspace{0mm} \\

\noindent
Kazushi Ueda

Department of Mathematics,
Graduate School of Science,
Osaka University,
Machikaneyama 1-1,
Toyonaka,
Osaka,
560-0043,
Japan.

{\em e-mail address}\ : \ kazushi@math.sci.osaka-u.ac.jp
\ \vspace{0mm} \\

\end{document}